\documentclass[11pt, reqno]{amsart}
\usepackage{amssymb,amscd,amsthm,amsxtra}
\usepackage{latexsym}

\usepackage{graphicx}  \usepackage{epstopdf}

\usepackage{color}
\usepackage[usenames,dvipsnames,svgnames,table]{xcolor}
\bibliography{refs}

\usepackage{hyperref}
\hypersetup{colorlinks = false}

\usepackage[mathscr]{euscript}
\renewcommand{\mathcal}[1]{{\mathscr#1}}

\newtheorem{theorem}{Theorem}[section]

\newtheorem{lemma}[theorem]{Lemma}
\newtheorem{prop}[theorem]{Proposition}

\theoremstyle{definition}
\newtheorem{definition}[theorem]{Definition}
\theoremstyle{remark}
\newtheorem{remark}[theorem]{Remark}
\numberwithin{equation}{section}

\newcommand{\R}{{\mathbb R}}
\newcommand{\N}{{\mathbb N}}

\renewcommand{\leq}{\leqslant}

\renewcommand{\geq}{\geqslant}

\renewcommand{\epsilon}{\varepsilon }

\newlength{\defbaselineskip}
\newcommand{\setlinespacing}[1]
           {\setlength{\baselineskip}{#1 \defbaselineskip}}

\author[B. Barrios]{Bego\~na Barrios}
\address[Bego\~na Barrios]{Departamento de An\'{a}lisis Matem\'{a}tico, Universidad de La Laguna
\hfill \break \indent C/Astrof\'{\i}sico Francisco S\'{a}nchez s/n, 38271 -- La Laguna, Spain}
\email{bbarrios@ull.es}

\author[M. Medina]{Maria Medina}
\address[Maria Medina]{Facultad de Matem\'aticas, Pontificia Universidad Cat\'olica de Chile, Avenida Vicu\~na Mackenna 4860, Santiago, Chile}
\email{mamedinad@mat.puc.cl}

\begin{document}

\subjclass[2010]{35B50, 35S15, 47G20, }

\keywords{Non local operators, fractional Laplacian, mixed boundary conditions, maximum principle}

\thanks{The first author was partially supported by a MEC-Juan de la Cierva postdoctoral fellowship  FJCI-2014-20504 (Spain). The second author was supported by the grant FONDECYT Postdoctorado, No. 3160077 (Chile).}

\title[Maximum principles]{Strong maximum principles for fractional elliptic and parabolic problems with mixed boundary conditions}

\begin{abstract}
We present some comparison results for solutions to certain non local elliptic and parabolic problems that involve the fractional Laplacian operator and mixed boundary conditions, given by a zero Dirichlet datum on part of the complementary of the domain and zero Neumann data on the rest. These results represent a non local generalization of a Hopf's lemma for elliptic and parabolic problems with mixed conditions. In particular we prove the non local version of the results obtained by J. D\'avila and J. D\'avila-L. Dupaigne for the classical case $s=1$ in \cite{D} and \cite{DD} respectively.
\end{abstract}

\maketitle


\section{Introduction}

The aim of this work is to study some comparison results for a class of elliptic and parabolic problems that involve the fractional Laplacian operator. More precisely we will consider the following non local, elliptic and parabolic, mixed problems,

\begin{equation}\label{prob}
\left\{
\begin{aligned}
(-\Delta)^s u&=f &&  \mbox{in } \Omega,\\
{u}&\geq0&& \mbox{in } \mathbb{R}^{N},\\
u&=0 && \mbox{in } \Sigma_1,\\
\mathcal{N}_s u&=0 &&\mbox{in }\Sigma_2,
\end{aligned}\right.
\end{equation}
\begin{equation}\label{probP}
\left\{
\begin{aligned}
u_t+(-\Delta)^s u&={0} &&  \mbox{in } \Omega\times (0,+\infty),\\
{u}&\geq0&& \mbox{in } \mathbb{R}^{N},\\
u&=0 && \mbox{in } \Sigma_1\times(0,+\infty),\\
\mathcal{N}_s u&=0 &&\mbox{in }\Sigma_2\times(0,+\infty),\\
u(x,0)&=u_0(x) && \mbox{ in }\Omega.
\end{aligned}\right.
\end{equation}
Here $\Omega$ is a bounded domain of $\R^N$, $\Sigma_1$ and $\Sigma_2$ are two open sets of positive measure satisfying
\begin{equation}\label{sigmas}
\Sigma_1\cap\Sigma_2=\emptyset,\qquad \overline{\Sigma_1}\cup\overline{\Sigma_2}=\R^N\setminus\Omega,
\end{equation}
$f\in\mathcal{C}^{\infty}_{0}(\Omega)$, $f\gneqq 0$ and $u_0\geq 0$, $u_0\in L^{2}(\Omega)$. The operator $(-\Delta)^s$, $0<s<1$, is the well-known \emph{fractional laplacian}, which
is defined on smooth functions as
\begin{equation}\label{operador}
(-\Delta)^s u(x) =  a_{N,s}\int_{\R^N} \frac{u(x)-u(y)}{|x-y|^{N+2s}} dy,
\end{equation}
where $a_{N,s}$ is a normalization constant that is usually omitted for brevity. The integral in \eqref{operador} 
has to be understood in the principal value sense, that is, as the limit as $\epsilon\to 0$ of the 
same integral taken in {$\R^N\setminus B_\epsilon(x)$, i.e, the complementary of the ball of center $x$ and radius $\epsilon$}.  See for instance \cite{guida, PhS, Stein} for the basic properties of the operator and the normalization constant. Problems with non local diffusion that involve the fractional Laplacian operator, and other integro-differential operators, have been intensively studied in the last years since they appear when we try to model different physical situations as anomalous diffusion and quasi-geostrophic flows, turbulence and
water waves, molecular dynamics and relativistic quantum mechanics of stars
(see \cite{BoG,CaV,Co} and references). They also appear in mathematical finance (cf. \cite{A,Be,CoT}), elasticity problems \cite{signorini}, obstacle problems \cite{BFR15,BFR16, CF}, phase transition \cite{AB98, SV08b} and
crystal dislocation \cite{dfv, toland} among others.

By  $\mathcal{N}_s$ we denote the non local normal derivative, defined as
\begin{equation}\label{neumann}
\mathcal{N}_su(x):= a_{N,s}\int_\Omega\frac{u(x)-u(y)}{|x-y|^{N+2s}}\,dy,\qquad x\in\R^N\setminus\overline{\Omega}.
\end{equation}
This function was introduced by S. Dipierro, X. Ros-Oton and E. Valdinoci in \cite{DRV} where the authors proved that, when $s\to 1^{-}$, the classical Neumann boundary condition $\frac{\partial u}{\partial\nu}$ is recovered in some sense. Moreover they established a complete description of the eigenvalues of $(-\Delta)^s$ with zero non local Neumann boundary condition, an existence and uniqueness result for the elliptic problem and the main properties of the fractional heat equation (preservation of mass, decreasing energy and convergence to a constant when $t\to\infty$) with this type of boundary condition. It is fair to mention here that other Neumann type boundary conditions for the non local problems, that recover the classical one when the fractional parameter $s$ goes to $1$, have been considered in the literature (see for instance \cite{barles, bogdan, cortazar}). 

Nevertheless the one given by \eqref{neumann} allows us to work in a variational framework and, as the authors described in \cite[Section 2]{DRV}, also has a natural probabilistic interpretation that we summarize here to motivate  the study of the elliptic problem \eqref{prob} for a general Dirichlet condition: let $\Omega\subseteq\mathbb{R}^{N}$ be  a bounded domain whose complementary is divided in two parts, satisfying \eqref{sigmas}, such that in $\Sigma_1$ there is a Dirichlet condition $h$ and one of Neumann type in $\Sigma_2$.  Let us now consider a particle that is randomly moving starting at a point $x_0\in\Omega$. There are two possibilities; if the particle goes to $x_1\in\Sigma_1$ then a payoff is obtained, established by the Dirichlet condition $h$, and if it goes to $x_2\in \Sigma_2$ then  immediately comes back to  some $y\in\Omega$  with a probability that is proportional to $|x_2-y|^{-N-2s}$. It is clear that the previous situation can be written as follows
$$u(x)=h(x),\quad x\in\Sigma_1,$$
$$u(x)={c(x)}\int_{\Omega}{u(y)|x-y|^{-N-2s}\, dy},\quad x\in\Sigma_2.$$
Choosing ${c(x)}$ in order to normalize the probability measure, that is,
$${c(x)}\int_{\Omega}{|x-y|^{-N-2s}\, dy}=1,$$ we finally get this behavior can be written as
$$\mathcal{N}_su(x)=0,\quad x\in\Sigma_2,$$
where $\mathcal{N}_s$ was given in \eqref{neumann}.

\medskip

Our motivation to study problem \eqref{prob} comes also from the fact that, as in the local case, by comparison one easily gets that there exists $C=\max_{\Omega} f$ such that 
$$u(x)\leq C v(x),\quad x\in\Omega,$$
where $v$ is the solution of \eqref{prob}  with $f=1$. However, it is not clear whether the opposite inequality 
\begin{equation}\label{opposite}
v(x)\leq  \widetilde{C}u(x),\quad x\in\Omega,
\end{equation}
is also true. We point out here that in the case of the Dirichlet  problem ($\Sigma_1=\mathbb{R}^{N}\setminus\Omega$), the previous estimate is obtained using the Hopf's Lemma and the $\mathcal{C}^{s}$ regularity of the solutions up to the boundary (see \cite{quim}). In the local case ($s=1$) in \cite{D} J. D\'avila proved that \eqref{opposite} also holds for the mixed problem with a constant $\widetilde{C}$ that depends on $\|f v\|_{L^{1}(\Omega)}$. Here, adapting the arguments to the non local framework, we obtain the same type of result for the fractional elliptic problem with mixed boundary conditions (see Theorem \ref{main_elliptic} below). Moreover, generalizing some results of \cite{DD}, we also get the desired inequality in the parabolic case (see Theorem \ref{main_parabolic} below). It is remarkable to point out that an inequality like \eqref{opposite} would be very useful for example in the study of certain nonlinear problems such as, for instance, mixed problems with concave-convex nonlinearities with critical growth because, due to the lack of regularity up to the boundary of the domain, a suitable space to separate solutions is needed (see \cite{colorado} and the references therein for the case $s=1$). 

\medskip

To conclude this section let us state the main two results of this paper, which are the non local counterpart of \cite[Theorem 1]{D} and \cite[Theorem 2.15]{DD} respectively. 
Consider $\overline{\xi_0}$ the solution to

\begin{equation}\label{probV}
\left\{
\begin{aligned}
(-\Delta)^s \overline{\xi_0}&=1 &&  \mbox{in } \Omega,\\
\overline{\xi_0}&=0 && \mbox{in } \Sigma_1,\\
\mathcal{N}_s \overline{\xi_0}&=0 &&\mbox{in }\Sigma_2.
\end{aligned}\right.
\end{equation}

\noindent Thus, the following maximum principles hold:

\begin{theorem}\label{main_elliptic}
Let $u$ be the solution to \eqref{prob} with $f\in\mathcal{C}_0^\infty(\Omega)$, $f\geq 0$, and let $\overline{\xi_0}$ be the solution to \eqref{probV}. Then there exists a constant $c=c({N,s,}\Omega, \Sigma_1, \Sigma_2)>0$ such that
$$u(x)\geq c\left(\int_\Omega f(y)\overline{\xi_0}(y)\,dy\right)\overline{\xi_0}(x),\qquad x\in\Omega.$$
\end{theorem}

\begin{theorem}\label{main_parabolic}
Let $u$ be the solution to \eqref{probP} with $u_0\in L^2(\Omega)$, $u_0\geq 0$, and let $\overline{\xi_0}$ solve \eqref{probV}. Then,
\begin{equation}\label{estimacion_parabolica}
u(x,t)\geq c(t)\left(\int_\Omega u_0(y)\overline{\xi_0}(y)\,dy\right)\overline{\xi_0}(x),\qquad (x,t)\in\Omega\times (0,+\infty),
\end{equation}
where $c(t)$ depends on $N$, $s$, $\Omega$, $\Sigma_1$ and $\Sigma_2$, {and is positive for $t>0$}.
\end{theorem}

The rest of the paper is organized as follows: in Section 2 we give some preliminaries related to the functional framework associated to problems \eqref{prob}-\eqref{probP} and we introduce the notion of solutions that will be used along the work. Section 3 deals with the proof of Theorem \ref{main_elliptic}. Finally in Section 4 we obtain the proof of Theorem \ref{main_parabolic}. 

We remark here that along the work we will denote by $C$ a positive constant that may change from line to line.

\section{Functional Setting and main results}

Let $u,v:\R^N\rightarrow \R$ be measurable functions and denote $Q:=\R^{2N}\setminus (\mathcal{C}\Omega)^2$. Consider the scalar product
\begin{equation}\label{scalprod}
\langle u,v\rangle_{E_{\Sigma_1}^s}:=\int_\Omega uv\,dx+\iint_Q\frac{(u(x)-u(y))(v(x)-v(y))}{|x-y|^{N+2s}}\,dxdy,
\end{equation}
and the associated norm
\begin{equation*}
\|u\|_{E_{\Sigma_1}^s}^2:=\int_\Omega u^2\,dx+\iint_Q\frac{|u(x)-u(y)|^2}{|x-y|^{N+2s}}\,dxdy.
\end{equation*}
Thus, we define the space
\begin{equation*}
E_{\Sigma_1}^s:=\{u:\R^N\rightarrow \R\mbox{ measurable s.t. }\|u\|_{E_{\Sigma_1}^s}<+\infty\mbox{ and }u=0\mbox{ in }\Sigma_1\}.
\end{equation*}

\begin{prop}
$E_{\Sigma_1}^s$ is a Hilbert space with the scalar product defined in \eqref{scalprod}.
\end{prop}

\begin{definition}
Let $f\in L^2(\Omega)$. We say that $u\in E_{\Sigma_1}^s$ is a weak solution of \eqref{prob} if
$$\frac{a_{N,s}}{2}\iint_Q\frac{(u(x)-u(y))(\varphi(x)-\varphi(y))}{|x-y|^{N+2s}}\,dxdy = \int_\Omega f\varphi\,dx,
$$
for every $\varphi\in E_{\Sigma_1}^s$.
\end{definition}

\begin{remark}
Notice that the domain of integration in the left hand side naturally arises from the problem, even when $u$ does not vanish in the whole $\R^N\setminus\Omega$. Indeed, multiplying in \eqref{prob} by a smooth function $\varphi\in E_{\Sigma_1}^s$ and integrating in $\Omega$ we get
\begin{eqnarray*}
&&\frac{a_{N,s}}{2}\iint_Q\frac{(u(x)-u(y))(\varphi(x)-\varphi(y))}{|x-y|^{N+2s}}\,dxdy\\
&=& \int_{\Omega}{\varphi(-\Delta)^su\, dx}+\int_{\mathbb{R}^{N}\setminus\Omega}{\varphi \mathcal{N}_su\, dx}\\
&=&\int_\Omega f\varphi\,dx.
\end{eqnarray*}
\end{remark}

\noindent We can also establish a Poincar\'e type inequality for this space with mixed conditions.

\begin{prop}\label{poincare}(Poincar\'e inequality) There exists a constant $C=C(\Omega,N,s)>0$ such that
$$\int_\Omega u^2\,dx\leq C\iint_Q\frac{|u(x)-u(y)|^2}{|x-y|^{N+2s}}\,dxdy,$$
for every $u\in E_{\Sigma_1}^s$. In particular, this implies the positivity of the first eigenvalue of the elliptic problem with zero mixed conditions, that is, $\lambda_1>0$ with
\begin{equation}\label{probEVm}
\left\{
\begin{aligned}
(-\Delta)^s \overline{\chi}_1&=\lambda_1 \overline{\chi}_1 &&  \mbox{in } \Omega,\\
\overline{\chi}_1&=0 && \mbox{in } \Sigma_1,\\
\mathcal{N}_s \overline{\chi}_1&=0 &&\mbox{in }\Sigma_2.
\end{aligned}\right.
\end{equation}
\end{prop}
\begin{proof}
Let us denote
$$\langle\varphi,\phi\rangle_{X^s(\Omega)}:=\frac{a_{N,s}}{2}\iint_Q\frac{(\varphi(x)-\varphi(y))(\phi(x)-\phi(y))}{|x-y|^{N+2s}},$$
and $[\varphi]_{X^s(\Omega)}:=\langle \varphi,\varphi\rangle_{X^s(\Omega)}^{1/2},$ where $\varphi,\phi\in E_{\Sigma_1}^s$. 

With this notation, it is clear that $\|\cdot\|_{E_{\Sigma_1}^s}^2=\|\cdot\|_{L^2(\Omega)}^2+[\cdot]_{X^s(\Omega)}^2$, and thus, we want to prove that
\begin{equation*}\begin{split}
\lambda_1:=&\displaystyle\inf_{u\in E_{\Sigma_1}^s}\frac{[u]_{X^s(\Omega)}^2}{\|u\|_{L^2(\Omega)}^2}=\inf_{u\in E_{\Sigma_1}^s,\, \|u\|_{L^2(\Omega)}=1}[u]_{X^s(\Omega)}^2>0.
\end{split}\end{equation*}
We proceed by contradiction. Suppose $\lambda_1=0$. Hence, one can find a sequence $\{u_k\}_{k\in\N}\in E_{\Sigma_1}^s$ such that
\begin{equation}\label{contradP}
\|u_k\|_{L^2(\Omega)}=1, \qquad [u_k]_{X^s(\Omega)}\rightarrow 0\mbox{ as }k\rightarrow +\infty.
\end{equation}
In particular, for $k$ large enough there exists a constant such that $\|u_k\|_{E_{\Sigma_1}^s}\leq C$. Thus, by \cite[Theorem 7.1]{guida},
$$u_k\rightharpoonup u \mbox{ in }E_{\Sigma_1}^s,\qquad u_k\rightarrow u \mbox{ in }L^2(\Omega),$$
and hence $\langle u_k-u,\varphi\rangle_{X^s(\Omega)}\rightarrow 0$ for every $\varphi\in E_{\Sigma_1}^s$ and
\begin{equation}\label{normaL2}
\|u\|_{L^2(\Omega)}=1.
\end{equation}
Taking $\varphi=u$ we obtain
$$\langle u_k-u,u\rangle_{X^s(\Omega)}\rightarrow 0, \mbox{ i.e. }\langle u_k,u\rangle_{X^s(\Omega)}\rightarrow \langle u,u\rangle_{X^s(\Omega)}.$$
Therefore, by the Cauchy-Schwartz inequality and \eqref{contradP}
$$
[u]_{X^s(\Omega)}^2=\lim_{k\rightarrow\infty}\langle u_k,u\rangle_{X^s(\Omega)}\leq [u]_{X^s(\Omega)}\left(\lim_{k\rightarrow\infty} [u_k]_{X^s(\Omega)}\right)=0.
$$
Thus according to the definition of $[\cdot]_{X^s(\Omega)}$, this implies that $u$ is constant in $\R^N$. But we know that $u=0$ in $\Sigma_1$, and hence $u=0$ in the whole $\R^N$, which contradicts  \eqref{normaL2}.
\end{proof}

It worths to point out here that the seminorm given by the double integral 
\begin{equation}\label{doubleInt}
\iint_Q\frac{|u(x)-u(y)|}{|x-y|^{N+2s}}\,dxdy,
\end{equation}
is actually a norm when we impose zero mixed boundary conditions. Indeed, as we used in the previous proof, if this integral vanishes then necessarily $u$ has to be constant in the whole $\R^N$ and, since we know that $u=0$ in $\Sigma_1$, we conclude that it vanishes a.e. in $\R^N$. In this sense, our boundary conditions behave as Dirichlet conditions and, thanks to Proposition \ref{poincare}, we have the {analogous} between the norms $\|\cdot\|_{E_{\Sigma_1}^s}$ and $\|\cdot\|_{H_0^s(\Omega)}$, defined as the double integral given in \eqref{doubleInt}, when one imposes $u=0$ in $\R^N\setminus\Omega$ to the functions in $H^s(\Omega)$. See for instance \cite{guida} for more details about these fractional Sobolev spaces.

As a consequence of Proposition \ref{poincare}, the coercivity of the operator in $E_{\Sigma_1}^s$ holds and Lax-Milgram theorem can be applied to guarantee the existence and uniqueness of solution of \eqref{prob} when $f\in L^2(\Omega)$. Likewise, one can assure the solvability of \eqref{probP} when $u_0\in L^2(\Omega)$. Moreover, to be consistent with the notation and the concept of solution introduced by J. D\'avila and L. Dupaigne in \cite{DD}, we will use the notion of analytic semigroup to give the precise definition of solutions to problem \eqref{probP}.

\begin{definition}
Let $\{S(t)\}_{t\geq 0}$ be the analytic semigroup in $L^{2}(\Omega)$ for the heat fractional equation with mixed boundary conditions. Then for every $u_{0}\in L^{2}(\Omega)$ there exists a unique
$$u:=S(t)u_0\in\mathcal{C}([0,\infty); L^{2}(\Omega))\cap \mathcal{C}((0,\infty); E_{\Sigma_1}^s)\cap \mathcal{C}^1((0,\infty); L^{2}(\Omega)),$$
solving \eqref{probP}. In particular, $u$ satisfies 
$$\int_\Omega\!\!u_t(x,\tau)\varphi(x,\tau)\,dx+\frac{a_{N,s}}{2}\!\!\iint_Q\frac{(u(x,\tau)-u(y,\tau))(\varphi(x,\tau)-\varphi(y,\tau))}{|x-y|^{N+2s}}\,dxdy=0$$
for every $\tau>0$ and $\varphi\in \mathcal{C}((0,\infty); E_{\Sigma_1}^s)$.
\end{definition}

Notice that the regularity properties follow in a standard way from the hilbertian structure of the space $E_{\Sigma_1}^s$ (see for instance \cite[Theorem 2.3 and Proposition 2.20]{anibal} or \cite[$\S$5.9.2, Theorem 3]{Ev}).

\begin{remark}\label{comparison}
Comparison results can be proved in both elliptic and parabolic cases with standard arguments, so we omit the proofs and the precise statements, but they will be often used along the work. Furthermore, we will frequently use the fact that if $u$ is a positive solution to a mixed problem, and $\tilde{u}$ is a positive solution to the analogous Dirichlet problem, then necessarily $u\geq \tilde{u}$, which follows straightforward from the comparison results for Dirichlet problems. This in particular implies that (see \cite{bmp})
$$u(z)\geq C\left(\int_{\Omega}{f(x)\delta^{s}(x)\, dx}\right)\delta^{s}(z),\, C>0,$$
{where $\delta^s(x):=\mbox{dist}(x,\partial\Omega)$}. It is worthy to mention here that, as far as we know, by the lack of regularity the previous inequality does not directly imply the statement of Theorem \ref{main_elliptic} as occurs in the case of zero Dirichlet condition.

\end{remark}

Finally, along this work we will need to make use of the following Hardy-type inequality, that can be found in, for example, \cite{hardy}.
\begin{prop}\label{HardyIneq}
There exists a constant $C=C(N,s)>0$ such that for every $\varphi\in H_0^s(\Omega)$ the following inequality holds,
$$\int_\Omega\frac{\varphi^2(x)}{\delta^{2s}(x)}\,dx\leq C\iint_Q\frac{(\varphi(x)-\varphi(y))^2}{|x-y|^{N+2s}}\,dxdy,$$
where $\delta(x)=\textrm{dist}(x,\partial\Omega)$ denotes the Euclidean distance in $\R^N$ to the boundary.
\end{prop}

\section{Elliptic Maximum Principle}
The aim of this section is to prove Theorem \ref{main_elliptic} but, before that, we will need some auxiliary results as the following one that can be seen as a kind of weighted Sobolev inequality:
\begin{lemma}\label{Sob}
Let $u$ be the solution of \eqref{prob} with $f\in L^\infty(\Omega)$, $f\gneq 0$. Then, there exists $C>0$ such that for every $\varphi\in E_{\Sigma_1}^s$
$$\left(\int_\Omega u^r|\varphi|^q\,dx\right)^{1/q}\leq C\left(\!\iint_Q u(x)u(y)\frac{(\varphi(x)-\varphi(y))^2}{|x-y|^{N+2s}}\,dxdy +\int_\Omega u^2\varphi^2\,dx\right)^{\!1/2},$$
where $0\leq r\leq 2^*_s$, $\frac{q}{2}=1+r\frac{s}{N}$ and the constant $C$ depends on $\Omega$, $N$, $s$, $\|u\|_{L^\infty(\Omega)}$, $\|f\|_{L^\infty(\Omega)}$ and $1/(\int_\Omega f(y)\delta^s(y)\,dy)$.

\end{lemma}

\begin{remark} Lemma \ref{Sob} is crucial in the proofs of Theorem \ref{main_elliptic} and Theorem \ref{main_parabolic}, and can be seen as the non local version of \cite[Lemma 3]{D}. Notice that the term $\int_\Omega u^2|\nabla \varphi|^2\,dx$ appearing there is replaced in this case by the non local term
\begin{equation}\label{buuu}
\iint_Q u(x)u(y)\frac{(\varphi(x)-\varphi(y))^2}{|x-y|^{N+2s}}\,dxdy,
\end{equation}
whose precise form is way less clear than in the local case. Indeed, in the local problem this term naturally comes from testing in a problem where the main operator has the divergence form $-\mbox{div}(u^2|\nabla w|)$ for some concrete $w$. However, in the fractional case one cannot explicitly compute the problem satisfied by this $w$, so at the begining it is not evident at all how the estimate in Lemma \ref{Sob} has to be. During the proof it will become clear that \eqref{buuu} is the appropriate term in this case.
\end{remark}

\begin{proof}
We proceed as in the proof of \cite[Lemma 3]{D}, i.e., we first prove the inequality for $r=0$, then for $r=2^*_s$ and finally we interpolate to obtain the result.
\medskip

\noindent {\it Step 1:} Case $r=0$.

\noindent Let $\chi_1>0$ be the first eigenfunction of the fractional Laplacian with zero Dirichlet conditions, that is, the solution of
\begin{equation}\label{probEV}
\left\{
\begin{aligned}
(-\Delta)^s \chi_1&=\lambda_1\chi_1 &&  \mbox{in } \Omega,\\
\chi_1&=0 && \mbox{in } \mathbb{R}^N\setminus \Omega,
\end{aligned}\right.
\end{equation}
which, by the regularity result obtained in \cite[Proposition 1.1]{quim} and the Hopf's Lemma given in \cite[Lemma 3.2]{quim} (see also \cite[Proposition~2.7]{crs}), satisfies $c_1\delta^s\leq \chi_1\leq c_2\delta^s$ for positive constants $c_1$ and $c_2$, with $\delta(x):=\mbox{dist}(x,\partial\Omega)$. Thus by Proposition \ref{HardyIneq} and \eqref{probEV} we obtain
\begin{equation*}\begin{split}
\int_\Omega \varphi^2\,dx \leq&\, C\int_\Omega\frac{\varphi^2\chi_1^2}{\delta^{2s}}\,dx\leq C\iint_Q\frac{(\varphi(x)\chi_1(x)-\varphi(y)\chi_1(y))^2}{|x-y|^{N+2s}}\,dxdy\\
=&\,C\left(\iint_Q\frac{(\chi_1(x)-\chi_1(y))(\chi_1(x)\varphi^2(x)-\chi_1(y)\varphi^2(y))}{|x-y|^{N+2s}}\,dxdy\right.\\
&\left.+\iint_Q\chi_1(x)\chi_1(y)\frac{(\varphi(x)-\varphi(y))^2}{|x-y|^{N+2s}}\,dxdy\right)\\
=&\,C\left(\lambda_1\int_\Omega \chi_1^2\varphi^2\,dx+\iint_Q\chi_1(x)\chi_1(y)\frac{(\varphi(x)-\varphi(y))^2}{|x-y|^{N+2s}}\,dxdy\right)\\
\leq&\, C\left(\lambda_1\int_\Omega \delta^{2s}\varphi^2\,dx+\iint_Q\delta^s(x)\delta^s(y)\frac{(\varphi(x)-\varphi(y))^2}{|x-y|^{N+2s}}\,dxdy\right).
\end{split}\end{equation*}
Applying now that, by the Hopf's Lemma, $u\geq C\delta^s$, $C>0$, it follows that
\begin{equation}\label{r0}
\int_\Omega \varphi^2\,dx \leq C\left(\int_\Omega u^2\varphi^2\,dx+\iint_Qu(x)u(y)\frac{(\varphi(x)-\varphi(y))^2}{|x-y|^{N+2s}}\,dxdy\right),
\end{equation}
as wanted.
\medskip

\noindent {\it Step 2:} Case $r=2^*_s$.

\noindent Using the Sobolev inequality and the fact that $u$ solves \eqref{prob} it follows that
\begin{equation*}
\begin{split}
\left(\int_\Omega |u\varphi|^{2^*_s}\,dx\right)^{2/2^*_s}\leq&\, C\left(\iint_Q\frac{(u(x)\varphi(x)-u(y)\varphi(y))^2}{|x-y|^{N+2s}}\,dxdy+\int_\Omega{u^2\varphi^2\,dx}\right)\\
=&\,C\left(\iint_Q\frac{(u(x)-u(y))(u(x)\varphi^2(x)-u(y)\varphi^2(y))}{|x-y|^{N+2s}}\,dxdy\right.\\
&\left.+\iint_Qu(x)u(y)\frac{(\varphi(x)-\varphi(y))^2}{|x-y|^{N+2s}}\,dxdy+\int_\Omega u^2\varphi^2\,dx\right)\\
=&\,C\left(\int_\Omega fu\varphi^2\,dx+\iint_Qu(x)u(y)\frac{(\varphi(x)-\varphi(y))^2}{|x-y|^{N+2s}}\,dxdy\right.\\
&\left.+\int_\Omega u^2\varphi^2\,dx\right).
\end{split}\end{equation*}
Since by hypothesis $f\in L^\infty(\Omega)$ repeating verbatim the Moser's type proof done for fractional elliptic problems with zero boundary conditions (see \cite{Leonori} for the linear case and \cite{Stefano} for the nonlinear one) we get that $u\in L^\infty(\Omega)$. Thus, by the inequality \eqref{r0} obtained in Step 1 and the previous estimate it follows
\begin{equation}\label{r2}
\begin{split}
\left(\int_\Omega |u\varphi|^{2^*_s}\,dx\right)^{2/2^*_s}\leq&\, C\left(\int_\Omega \varphi^2\,dx+\iint_Qu(x)u(y)\frac{(\varphi(x)-\varphi(y))^2}{|x-y|^{N+2s}}\,dxdy\right.\\
&\left.+\int_\Omega u^2\varphi^2\,dx\right)\\
\leq &\, C\left(\iint_Qu(x)u(y)\frac{(\varphi(x)-\varphi(y))^2}{|x-y|^{N+2s}}\,dxdy+\int_\Omega u^2\varphi^2\,dx\right),
\end{split}
\end{equation}
and we conclude.
\medskip

\noindent {\it Step 3:} Interpolation.

\noindent Let be $0<\lambda<1$. By H\"older's inequality,
$$\int_\Omega u^r|\varphi|^q\,dx\leq \left(\int_\Omega \varphi^2\,dx\right)^{1-\lambda}\left(\int_\Omega u^{r/\lambda}|\varphi|^{(q-2(1-\lambda))/\lambda}\,dx\right)^\lambda.$$
Fixing now $\lambda$ so that 
$$\frac{r}{\lambda}=\frac{q-2(1-\lambda)}{\lambda}=2^*_s,$$
applying inequalities \eqref{r0} and \eqref{r2} we obtain that
$$\left(\int_\Omega u^r|\varphi| ^q\,dx\right)^{\!\!1/q}\!\leq C\!\left(\iint_Q u(x)u(y)\frac{(\varphi(x)-\varphi(y))^2}{|x-y|^{N+2s}}\,dxdy\!+\!\int_\Omega u^2\varphi^2\,dx\right)^{\!\!\frac{1-\lambda+\lambda \frac{2^*_s}{2}}{q}.}$$
Noticing that $(1-\lambda+\lambda \frac{2^*_s}{2})\frac{1}{q}=\frac{1}{2}$ we conclude.
\end{proof}

\begin{remark}
It can be easily seen that in the particular case of $u=\overline{\chi}_1$, the solution to \eqref{probEVm}, the proof simplifies and the constant depends only on $\Omega$, $N$, $s$ and $\lambda_1$.
\end{remark}

\noindent We introduce now the fundamental auxiliary result to prove Theorem \ref{main_elliptic}:

\begin{lemma}\label{aux_elliptic}
Let $v$ be the solution to he problem
\begin{equation}\label{probV2}
\left\{
\begin{aligned}
(-\Delta)^s v&=g &&  \mbox{in } \Omega,\\
v&=0 && \mbox{in } \Sigma_1,\\
\mathcal{N}_s  v&=0 &&\mbox{in }\Sigma_2,
\end{aligned}\right.
\end{equation}
with $g\in L^p(\Omega)$, $p>N/s$, and let $u$ be the solution of \eqref{prob} with $f\in L^\infty(\Omega)$, $f\gneq 0$. Then there exists a constant $C>0$ such that
$$\left\|\frac{v}{u}\right\|_{L^\infty(\Omega)}\leq C\|g\|_{L^p(\Omega)},$$
with $C$ depending on $\Omega$, $\Sigma_1$, $\Sigma_2$, $N$, $p$, $\|u\|_{L^\infty(\Omega)}$, $\|f\|_{L^\infty(\Omega)}$ and $\|f\delta^s\|^{-1}_{L^{1}(\Omega)}$.

\end{lemma}

\begin{proof}
First of all we point out here that, since $p>N/s>N/(2s)$, as we commented before, following the ideas developed in \cite{Leonori}, the function $v$ belongs to $L^{\infty}(\Omega)$. Let us now consider $\varphi\in {E_{\Sigma_1}^s}\cap L^\infty(\Omega)$. Thus, using $u\varphi$, $v\varphi$, that belong to $E_{\Sigma_1}^s$, as test functions in \eqref{probV2} and \eqref{prob} respectively, it follows that
\begin{equation}
\begin{split}
&\int_\Omega (gu-fv)\varphi\,dx\\
&=\frac{a_{N,s}}{2}\iint_Q\frac{(v(x)-v(y))(u(x)\varphi(x)-u(y)\varphi(y))}{|x-y|^{N+2s}}\,dxdy\\
&\;\;\;\;-\frac{a_{N,s}}{2}\iint_Q\frac{(u(x)-u(y))(v(x)\varphi(x)-v(y)\varphi(y))}{|x-y|^{N+2s}}\,dxdy\\
&=\frac{a_{N,s}}{2}\iint_Q\frac{u(x)v(y)\varphi(y)+u(y)v(x)\varphi(x)-v(x)u(y)\varphi(y)-v(y)u(x)\varphi(x)}{|x-y|^{N+2s}}\,dxdy\\
&= \frac{a_{N,s}}{2}\iint_Q\frac{v(x)u(y)-u(x)v(y)}{|x-y|^{N+2s}}(\varphi(x)-\varphi(y))\,dxdy \\
&=\frac{a_{N,s}}{2}\iint_Q\frac{u(y)(v(x)-v(y))-v(y)(u(x)-u(y))}{|x-y|^{N+2s}}(\varphi(x)-\varphi(y))\,dxdy.\label{eq6}
\end{split}\end{equation}
\noindent We take now $\epsilon>0$ and $k\geq 0$ and we set 
$$\displaystyle\varphi_\epsilon:=\left(\frac{v}{u+\epsilon}-k\right)_+\in {E_{\Sigma_1}^s}\cap L^\infty(\Omega).$$ 
We want to see that
\begin{equation}\begin{split}\label{blabla}
(u+\epsilon)(x)&(u+\epsilon)(y)(\varphi_\epsilon(x)-\varphi_\epsilon(y))^2\\
\leq &\left[u(y)(v(x)-v(y))-v(y)(u(x)-u(y))\right](\varphi_\epsilon(x)-\varphi_\epsilon(y))\\
&+\epsilon(v(x)-v(y))(\varphi_\epsilon(x)-\varphi_\epsilon(y)),
\end{split}\end{equation}
for $x$, $y$ in $\R^N$. If $x,y\in\{v\geq k(u+\epsilon)\}$ or $x,y\in \{v<k(u+\epsilon)\}$ the inequality easily follows (it is an identity indeed). Consider now the case $x\in\{v\geq k(u+\epsilon)\}$ and $y\in \{v<k(u+\epsilon)\}$. Thus, 
$$\varphi_\epsilon(x)=\frac{v(x)}{(u+\epsilon)(x)}-k,\qquad \varphi_\epsilon(y)=0,\mbox{ and }k\geq \frac{v(y)}{(u+\epsilon)(y)}.$$
Therefore,
\begin{equation*}\begin{split}
(u+\epsilon)(x)(&u+\epsilon)(y)(\varphi_\epsilon(x))^2=(u+\epsilon)(x)(u+\epsilon)(y)\left(\frac{v(x)}{(u+\epsilon)(x)}-k\right)\varphi_\epsilon(x)\\
=&\left\{(u+\epsilon)(y)v(x)-k(u+\epsilon)(x)(u+\epsilon)(y)\right\}\varphi_\epsilon(x)\\
\leq & \left\{(u+\epsilon)(y)v(x)-(u+\epsilon)(x)v(y)\right\}\varphi_\epsilon(x)\\
=&\left\{u(y)(v(x)-v(y))-v(y)(u(x)-u(y))+\epsilon(v(x)-v(y))\right\}\varphi_\epsilon(x),
\end{split}\end{equation*}
and \eqref{blabla} follows. Likewise, it holds whenever $y\in\{v\geq k(u+\epsilon)\}$ and $x\in \{v<k(u+\epsilon)\}$.

Hence, substituting in \eqref{eq6} with $\varphi=\varphi_\epsilon$, by \eqref{blabla} we obtain
\begin{equation*}\begin{split}\label{eq9}
\frac{a_{N,s}}{2}&\iint_Q (u+\epsilon)(x)(u+\epsilon)(y)\frac{(\varphi_\epsilon(x)-\varphi_\epsilon(y))^2}{|x-y|^{N+2s}}\,dxdy\\
&\,\leq\epsilon\frac{a_{N,s}}{2}\iint_Q \frac{(v(x)-v(y))(\varphi_\epsilon(x)-\varphi_\epsilon(y))}{|x-y|^{N+2s}}\,dxdy+\int_\Omega(gu-fv)\varphi_\epsilon\,dx,
\end{split}\end{equation*}
and using the positivity of $f$, $v$ and $\varphi_\epsilon$, and the fact that $v$ solves \eqref{probV2}, it yields
\begin{equation*}\label{eq11}
\frac{a_{N,s}}{2}\iint_Q (u+\epsilon)(x)(u+\epsilon)(y)\frac{(\varphi_\epsilon(x)-\varphi_\epsilon(y))^2}{|x-y|^{N+2s}}\,dxdy\leq \int_\Omega g(u+\epsilon)\varphi_\epsilon\,dx
\end{equation*}
Combining now this estimate with Lemma \ref{Sob} we get
\begin{equation}\label{star}
\left(\int_\Omega (u+\varepsilon)^r|\varphi_\epsilon|^q\,dx\right)^{2/q}\leq C\left(\int_\Omega g(u+\epsilon)\varphi_\epsilon\,dx+\int_\Omega (u+\varepsilon)^2\varphi_\epsilon^2\,dx\right),
\end{equation}
where $q=2+2rs/N$. Denoting $w:=v/u$, since
$$(u+\epsilon)\varphi_\epsilon=(v-k(u+\epsilon))_+\rightarrow (v-ku)_+=u(w-k)_+,$$
when $\epsilon\rightarrow 0$, from \eqref{star} we obtain, by monotone convergence, 
$$\left(\int_\Omega u^r(w-k)_+^q\,dx\right)^{2/q}\leq C\left(\int_\Omega gu(w-k)_+\,dx+\int_\Omega u^2(w-k)_+^2\,dx\right).$$
We choose now $r=\frac{p}{p-1}\in (1,2^*_s)$. Notice that in this case $q>2$ and $2\frac{q-r}{q-2}>0$. Thanks to this, the fact that the previous integral inequality is purely local allows us to conclude the proof exactly as in \cite[Lemma 2]{D} using an iterative Stampacchia method. We mention here that the necessity of requiring $g\in L^{p}(\Omega)$ with $p>N/s$ comes from this iterative method. In fact, to obtain the conclusion of the theorem is important to be able to affirm that the solution of some Bernoulli type differential inequality $A(k)\leq C\|g\|_{L^{p}(\Omega)}(-A'(k))^\gamma$, $\gamma=2-2/q-1/p$, is equal to zero for some $k(\|g\|^{1/\gamma}_{L^{p}(\Omega)})$, bigger than a fixed quantity that depends on $\|v\|_{L^{\infty}(\Omega)}$. For that $\gamma>1$ is needed, so the condition over $p$ comes out.
\end{proof}

\noindent Using the previous result and  following some ideas developed in \cite[Lemma 3.2]{brezcabre}, we are now able to give the
\begin{demoelliptic}
Let $K\subset\Omega$ be a fixed but arbitrary compact set strictly contained in $\Omega$. Then, there exists $r>0$ such that $r\leq \mbox{dist}(x_0,\partial\Omega)$ for every $x_0\in K$  so, by \cite[Proposition 2.2.6 and Proposition 2.2.2]{PhS}, it follows that
$$
u(x_0)\geq \int_{\R^N}{u(z)\gamma_{r}(z-x_0)\,dz}{\geq}\int_{\Omega}{u(z)\gamma_{r}(z-x_0)\,dz}>0,\, x_0\in K.
$$
Here $\gamma_{r}:=(-\Delta)^s\Gamma_{r}$ where $\Gamma_{r}$ is a $\mathcal{C}^{1,1}$ function that matches outside the ball $B(0,r)$ with the fundamental solution $\Phi:=C|x|^{2s-N}$ and that is a paraboloid inside this ball. Then there exists a positive constant $c>0$ such that $u(x_0)> c$ for every $x_0\in K$. That is
\begin{equation}\label{nyy}
u(x_0)>M\int_{\Omega}{u(z)\, dz},\quad x_0\in K,
\end{equation}
where
$$M=c\left(\int_{\Omega}{u(z)\, dz}\right)^{-1}>0.$$
Consider now the solution $w$ of 
$$
\left\{
\begin{aligned}
(-\Delta)^s w&=f _0&&  \mbox{in } \Omega,\\
w&=0 && \mbox{in } \Sigma_1,\\
\mathcal{N}_s w&=0 &&\mbox{in }\Sigma_2,
\end{aligned}\right.
$$
where $0\lneqq f_0\leq 1$, $f_0\in\mathcal{C}^{\infty}_{0}(K)$. Therefore, by \eqref{nyy} and Lemma \ref{aux_elliptic}, for every $x\in K$ we get that
\begin{equation}\begin{split}
u(x)&\geq M\int_{\Omega}{u(z) f_0(z)\, dz}=\frac{a_{N,s}}{2}\iint_Q\frac{(u(x)-u(y))(w(x)-w(y))}{|x-y|^{N+2s}}\,dxdy\\
&= M\int_{\Omega}{w(z) f(z)\, dz}\geq C_0\int_{\Omega}{f(z)\overline{\xi_0}(z) \, dz}\geq \lambda w(x),\label{inK}
\end{split}
\end{equation}
where
$$\lambda:=\frac{C_0}{\|w\|_{L^{\infty}({\Omega})}}\int_{\Omega}{f(z) \overline{\xi_0}(z)\, dz}.$$ 
Then it is clear that
$$
\left\{
\begin{aligned}
(-\Delta)^s (u-\lambda w)&=f \geq 0&&  \mbox{in } \Omega\setminus K,\\
u-\lambda w&\geq 0 && \mbox{in } \Sigma_1\cup K,\\
\mathcal{N}_s {(u-\lambda w)}&=0 &&\mbox{in }\Sigma_2,
\end{aligned}\right.
$$
and therefore, by comparison (see Remark \ref{comparison}), it follows that 
\begin{equation}\label{outK}
\mbox{$u-\lambda w\geq 0$ in $\Omega\setminus K$.}
\end{equation}
Thus, by Lemma \ref{aux_elliptic}, \eqref{inK} and \eqref{outK} we conclude that
$$u(x)\geq C\left(\int_{\Omega}{f(z) \overline{\xi_0}(z)\, dz}\right) \overline{\xi_0}(x),\quad x\in \Omega,$$
as desired.
\end{demoelliptic}

\section{Parabolic Maximum Principle}

As happened in the elliptic case, before proving Theorem \ref{main_parabolic} we need to establish some comparison results. The first one will provide us a pointwise comparison between the  first eigenfunction of the fractional Laplacian  and the solution of the elliptic mixed problem with right hand side equal to one:

\begin{prop}\label{bounds}
Let $\overline{\chi}_1$ be the first eigenfunction of $(-\Delta)^s$ with {mixed boundary conditions in $\Omega$, i.e., the solution to \eqref{probEVm}} with $L^2(\Omega)$-norm equal to one. Then, there exists a positive constant $C=C(\Omega,N,s, \Sigma_1, \Sigma_2)$ such that
\begin{equation}\label{order}
{C^{-1}\overline{\xi_0}\leq \overline{\chi}_1\leq C\overline{\xi_0}\mbox{ in }\Omega},
\end{equation}
where $\overline{\xi_0}$ is the solution to \eqref{probV}.
\end{prop}

\begin{proof}
To prove that there exists $C>0$ such that $\overline{\chi}_1\leq C\overline{\xi_0}$, we consider the function 
\begin{equation}\label{division1}
w:=\frac{\overline{\chi}_1}{\overline{\xi_0}}.
\end{equation} Thus, taking, for $j\geq 1$, $\overline{\xi_0}w^{2j-1}$ and $\overline{\chi}_1 w^{2j-1}$ as test functions in \eqref{probEVm} and \eqref{probV} respectively, and proceeding as in \eqref{eq6} we obtain
\begin{equation*}\begin{split}
\lambda_1\int_\Omega &\overline{\chi}_1 \overline{\xi_0} w^{2j-1}\,dx \,\geq \int_\Omega (\lambda_1\overline{\chi}_1 \overline{\xi_0}-\overline{\chi}_1) w^{2j-1}\,dx\\
&\, =\frac{a_{N,s}}{2}\iint_Q\frac{\overline{\xi_0}(y)(\overline{\chi}_1(x)-\overline{\chi}_1(y))}{|x-y|^{N+2s}}(w^{2j-1}(x)-w^{2j-1}(y))\,dxdy\\
&\,- \frac{a_{N,s}}{2}\iint_Q\frac{\overline{\chi}_1(y)(\overline{\xi_0}(x)-\overline{\xi_0}(y))}{|x-y|^{N+2s}}(w^{2j-1}(x)-w^{2j-1}(y))\,dxdy\\
&\,=\frac{a_{N,s}}{2}\iint_Q\overline{\xi_0}(x)\overline{\xi_0}(y)\frac{(w(x)-w(y))(w^{2j-1}(x)-w^{2j-1}(y))}{|x-y|^{N+2s}}\,dxdy.
\end{split}\end{equation*}
Applying now the numerical lemma \cite[Lemma 2.22]{AMPP} with $s_1:=w(x)$, $s_2:=w(y)$ and $a:=2j-1$ it yields
$$
\lambda_1\int_\Omega \overline{\chi}_1 \overline{\xi_0} w^{2j-1}\,dx \geq \left(\frac{2j-1}{{j^2}}\right)\frac{a_{N,s}}{2}\iint_Q\overline{\xi_0}(x)\overline{\xi_0}(y)\frac{(w^j(x)-w^j(y))^2}{|x-y|^{N+2s}}\,dxdy.
$$
Then, choosing now $u:=\overline{\xi_0}$ and $r:=2$ in Lemma \ref{Sob}, we conclude the existence of $q:=2\left(1+\frac{2s}{N}\right)$ and $C>0$ such that
\begin{equation}\begin{split}\label{iter}
C\left(\int_\Omega \overline{\xi_0}^2 w^{qj}\,dx\right)^{2/q}&\,\leq \lambda_1\int_\Omega \overline{\chi}_1 \overline{\xi_0} w^{2j-1}\,dx+ \widetilde{C}j\int_\Omega \overline{\xi_0}^2 w^{2j}\,dx\\
&\, \leq Cj\int_\Omega \overline{\xi_0}^2 w^{2j}\,dx.
\end{split}\end{equation}
If we define
$$\mu=\frac{q}{2},\qquad j_k:=2\mu^k,\qquad \theta_k:=\left(\int_\Omega\overline{\xi_0}^2w^{j_k}\right)^{1/j_k},\qquad k=0,1,\ldots$$
thus, \eqref{iter} can be rewritten as 
$$\theta_{k+1}\leq (C\mu^k)^{1/\mu^k}\theta_k,$$
and iterating we obtain that
$$\sup_\Omega w =\lim_{k\rightarrow\infty}\theta_k\leq C\theta_0=C\left(\int_\Omega\overline{\chi}_1^2\,dx\right)^{1/2}=C<+\infty.$$
Therefore, 
\begin{equation}\label{primera_ineq}
\overline{\chi}_1\leq C\overline{\xi_0}.
\end{equation} 
We notice here that to justify the computations above we can consider
$$w_\epsilon:=\frac{\overline{\chi}_1}{\overline{\xi_0}+\epsilon},$$
that is well defined in $\Omega$. Thus we can repeat the previous proof for the functions $w_{\epsilon}$ obtaining that $\sup_\Omega w_{\epsilon}\leq C$ and passing to the limit when $\epsilon\to 0$ to conclude. 
\medskip

To prove that $\overline{\xi_0}\leq C \overline{\chi}_1$ we consider 
\begin{equation}\label{division2}
w:=\frac{\overline{\xi_0}}{\overline{\chi_1}}. 
\end{equation}
Proceeding as before, and applying again \cite[Lemma 2.22]{AMPP} we obtain
$$\left(\frac{2j-1}{{j^2}}\right)\frac{a_{N,s}}{2}\iint_Q\overline{\chi_1}(x)\overline{\chi_1}(y)\frac{(w^j(x)-w^j(y))^2}{|x-y|^{N+2s}}\,dxdy\leq\int_\Omega\overline{\chi_1}w^{2j-1}\,dx.$$
Thus, \eqref{primera_ineq} implies
\begin{equation*}
\frac{a_{N,s}}{2}\iint_Q\overline{\chi_1}(x)\overline{\chi_1}(y)\frac{(w^j(x)-w^j(y))^2}{|x-y|^{N+2s}}\,dxdy\leq Cj\int_\Omega\overline{\chi_1}w^{2j}\,dx.
\end{equation*}
Applying H\"older's inequality, Lemma \ref{Sob} and Young's inequality on the right hand side of the previous inequality, we get that
\begin{equation*}\begin{split}
&\frac{a_{N,s}}{2}\iint_Q\overline{\chi_1}(x)\overline{\chi_1}(y)\frac{(w^j(x)-w^j(y))^2}{|x-y|^{N+2s}}\,dxdy\\
&\leq Cj\left(\int_\Omega\overline{\chi_1}^2w^{2j}\,dx\right)^{1/2}\left(\int_\Omega w^{2j}\,dx\right)^{1/2}\\
&\leq \frac{Cj^2}{2}\left(\int_\Omega\overline{\chi_1}^2w^{2j}\,dx\right)^{1/2}+\frac{a_{N,s}}{4}\iint_Q\overline{\chi_1}(x)\overline{\chi_1}(y)\frac{(w^j(x)-w^j(y))^2}{|x-y|^{N+2s}}\,dxdy\\
&\, \;\;\;\;+\frac{1}{2}\int_\Omega \overline{\chi_1}^2 w^{2j}\,dx.
\end{split}\end{equation*}
Therefore
\begin{equation*}
\frac{a_{N,s}}{2}\iint_Q\overline{\chi_1}(x)\overline{\chi_1}(y)\frac{(w^j(x)-w^j(y))^2}{|x-y|^{N+2s}}\,dxdy\leq Cj^2\int_\Omega\overline{\chi_1}^2w^{2j}\,dx.
\end{equation*}
Using this estimate together with Lemma \ref{Sob} we obtain
\begin{equation}\label{extra}
\left(\int_\Omega \overline{\chi_1}^2w^{qj}\,dx\right)^{2/q}\leq Cj^2\int_\Omega \overline{\chi_1}^2w^{2j}\,dx,
\end{equation}
with $q:=2\left(1+\frac{2s}{N}\right)$. Iterating as it was done in the proof of \eqref{primera_ineq}, we can conclude that $\sup_\Omega w\leq C$ and therefore 
\begin{equation}\label{segunda_ineq}
\overline{\xi_0}\leq C \overline{\chi_1}.
\end{equation}
We conclude noticing that, as in the proof of \eqref{primera_ineq}, the computations done to prove \eqref{segunda_ineq} can be justified considering $$w_\epsilon:=\frac{\overline{\xi_0}}{\overline{\chi_{1,\epsilon}}},$$
that is well defined in $\Omega$, whith $\overline{\chi_{1,\epsilon}}:=\overline{\chi_1}+\epsilon$. Repeating the previous estimates for the function $w_{\epsilon}$ we will get that 
$$\left(\int_\Omega \overline{\chi_{1,\epsilon}}^2w_{\epsilon}^{qj}\,dx\right)^{2/q}\leq Cj^2\int_\Omega \overline{\chi_{1,\epsilon}}^2w_{\epsilon}^{2j}\,dx + Cj\epsilon\int_\Omega \frac{\overline{\chi_{1,\epsilon}}}{\overline{\xi_0}}\, dx.$$ 
Thus, by the monotone convergence theorem, we can pass to the limit when $\epsilon\to 0$ achieving \eqref{extra}. The integrals that appear here are well defined due to Theorem \ref{main_elliptic}.
\end{proof}

\begin{remark}
The inequality \eqref{primera_ineq} can be proved by a simple comparison argument just by noticing that $\overline{\chi}_1\in L^\infty(\Omega)$ (that follows exactly as in \cite[Propostition 2.2]{BCSS}). We notice that in this case the inequality will be obtained with a constant depending on $\|\overline{\chi}_1\|_{L^{\infty}(\Omega)}$ . However, we keep the iterative proof since it can be applied to more general eigenvalue problems (for instance with unbounded potentials like in \cite{DD}).
\end{remark}

\noindent Now we are able to prove the next
\begin{theorem}\label{key_parabolico}
 Let $u$ and $\overline{\xi_0}$ be the solutions to \eqref{probP} and \eqref{probV} respectively. Then,
$$u(t)\geq c(t)\overline{\xi_0},$$
for some $c(t)>0$ depending also on $N$, $s$ and $\Omega$.
\end{theorem}

\begin{proof}
First of all we notice that, since by Proposition \ref{bounds} we know $\overline{\xi_0}\leq C{\overline{\chi}_1}$ (where ${\overline{\chi}_1}$ is the normalized solution of {\eqref{probEVm}} and $C=C(\Omega,N,s, \Sigma_1, \Sigma_2$)), the result holds if we prove 
\begin{equation}\label{objetivo}
u(t)\geq c(t){\overline{\chi}_1}.
\end{equation}
Let {$T>0$} and consider 
$$v(x,t):=e^{-\lambda_1 t}{\overline{\chi}_1}(x),\quad x\in\Omega,\quad 0<t<T,$$
that clearly satisfies
\begin{equation}\label{probVV}
\left\{
\begin{aligned}
v_t+(-\Delta)^s v&=0 &&  \mbox{in } \Omega\times (0,T),\\
v&=0 && \mbox{in } \Sigma_1\times (0,T),\\
\mathcal{N}_s v&=0 && \mbox{in } \Sigma_2\times (0,T),
\end{aligned}\right.
\end{equation}
We define now 
$$w(x,t):=\frac{v}{u}\quad\mbox{ and }\quad\theta_j(x,t):=\int_\Omega u^2(x,t)w^j(x,t)\,dx,\;\; j\geq 1,\;\; t\in (0,T).$$
In order to get \eqref{objetivo} our next goal is, using an iterative argument that involve the functions $\theta_j(x,t)$, to prove that
\begin{equation}\label{objetivo2}
w(x,t)\leq C_0 t^{-\beta},\qquad 0\leq t\leq T,
\end{equation}
for some $C_0>0$, $\beta>0$ {independent of $t$}.
From now on, when there is no possible confusion, we will omit the dependence of every function on the variable $t$ to simplify the notation.
To obtain \eqref{objetivo2} we notice that, by definition,
$$w_t=\frac{v_tu-vu_t}{u^2},$$
and from here, since $u$ and $v$ solve \eqref{probP} and \eqref{probVV} respectively, it can be seen that
$$\int_\Omega u^2\varphi w_t\,dx+\int_{\mathbb{R}^N}u\varphi(-\Delta)^sv\,dx-\int_{\mathbb{R}^N}v\varphi (-\Delta)^su\,dx=0.$$
Choosing $\varphi=w^{2j-1}$ and writing the weak formulation, this implies
\begin{equation*}\begin{split}
0=& \,\frac{1}{2j}\int_\Omega u^2(w^{2j})_t\,dx\\
&+\frac{a_{N,s}}{2}\iint_Q\frac{u(x)w^{2j-1}(x)-u(y)w^{2j-1}(y))(v(x)-v(y))}{|x-y|^{N+2s}}\,dxdy\\
&-\iint_Q\frac{(v(x)w^{2j-1}(x)-v(y)w^{2j-1}(y))(u(x)-u(y))}{|x-y|^{N+2s}}\,dxdy\\
=& \,\frac{1}{2j}\int_\Omega u^2(w^{2j})_t\,dx +\frac{a_{N,s}}{2}\iint_Q\frac{u(x)v(y)(w^{2j-1}(y)-w^{2j-1}(x))}{|x-y|^{N+2s}}\,dxdy\\
&\,+\frac{a_{N,s}}{2}\iint_Q \frac{v(x)u(y)(w^{2j-1}(x)-w^{2j-1}(y))}{|x-y|^{N+2s}}\,dxdy\\
=& \,\frac{1}{2j}\int_\Omega u^2(w^{2j})_t\,dx\\
&+\frac{a_{N,s}}{2}\iint_Qu(x)u(y)\frac{(w^{2j-1}(x)-w^{2j-1}(y))(w(x)-w(y))}{|x-y|^{N+2s}}\,dxdy.
\end{split}\end{equation*}
Applying once again \cite[Lemma 2.22]{AMPP}, it follows that
\begin{equation}\label{eq54}
\!\!\int_\Omega \!u^2(w^{2j})_t\,dx+\frac{a_{N,s}(2j-1)}{j}\!\!\iint_Q u(x)u(y)\frac{(w^j(x)-w^j(y))^2}{|x-y|^{N+2s}}\,dxdy\leq 0.
\end{equation} 
Moreover, since
$$\theta'_{2j}(t)=2\int_\Omega uu_tw^{2j}\,dx+\int_\Omega u^2(w^{2j})_t\,dx,$$
plugging this equality into \eqref{eq54}, we get that
\begin{equation}\label{eq56}
\theta'_{2j}-2\int_\Omega uu_tw^{2j}\,dx+\frac{a_{N,s}(2j-1)}{j}\!\iint_Qu(x)u(y)\frac{(w^j(x)-w^j(y))^2}{|x-y|^{N+2s}}\,dxdy \leq 0.
\end{equation}
Furthermore, since testing in \eqref{probP} with $uw^{2j}$ one gets
\begin{equation*}\begin{split}
\int_\Omega uu_tw^{2j}\,dx &=-\frac{a_{N,s}}{2}\iint_Q\frac{(u(x)-u(y))(u(x)w^{2j}(x)-u(y)w^{2j}(y))}{|x-y|^{N+2s}}\,dxdy\\
&=-\frac{a_{N,s}}{2}\iint_Q\frac{(u(x)w^{{j}}(x)-u(y)w^{{j}}(y))^2}{|x-y|^{N+2s}}\,dxdy\\
&\;\;\;\;+\frac{a_{N,s}}{2}\iint_Q u(x)u(y)\frac{(w^j(x)-w^j(y))^2}{|x-y|^{N+2s}}\,dxdy,
\end{split}\end{equation*}
by \eqref{eq56} we conclude
\begin{equation*}\begin{split}
\theta'_{2j}&+a_{N,s}\iint_Q\frac{(u(x)w^j(x)-u(y)w^j(y))^2}{|x-y|^{N+2s}}\,dxdy\\
&+\frac{a_{N,s}(2j-1)}{j}\iint_Qu(x)u(y)\frac{(w^j(x)-w^j(y))^2}{|x-y|^{N+2s}}\,dxdy\leq 0.
\end{split}\end{equation*}
Therefore we have obtained that $\theta'_{2j}(t)\leq 0$, $j\geq 1$, $0<t<T$, and this in particular implies 
\begin{equation}\label{imp1}
\theta_j(t)\leq \theta_j(0)\mbox{ for all }t\in [0,T] \mbox{ and } j\geq 2.
\end{equation}

On the other hand, by comparison with the solution of the fractional heat equation with zero Dirichlet condition and the Hopf's Lemma (see {\cite{bonforte, CKS}}) we have that
$$u(t)\geq c(t) \delta^s,$$
for some positive function $c(t)$. Thus, we can assume 
$$u(t)\geq c\delta^s \mbox{ for }t\in [0,T],$$
with $c>0$ independent of $t$ in this range, and we can proceed as in the proof of Lemma \ref{Sob} (see \eqref{r0}) to get
\begin{equation}\label{eq51}
\int_\Omega \varphi^2\,dx\leq C\left(\iint_Qu(x,t)u(y,t)\frac{(\varphi(x)-\varphi(y))^2}{|x-y|^{N+2s}}\,dxdy\!+\!\int_\Omega u^2\varphi^2\,dx\right)
\end{equation}
for every $\varphi\in E_{\Sigma_1}^s$, $t\in [0,T]$ and $C>0$ independent of $t$.

Therefore using \eqref{imp1} and \eqref{eq51} we can follow analogously to the proof of \cite[Claim 5.3]{DD} to get
$$
\frac{1}{C}\frac{\theta_{2j}^{1+\gamma}(t)}{\theta_j^{2\gamma}(0)}+\theta'_{2j}(t)\leq \theta_{2j}(t)
$$
with $\gamma:= \frac{2s}{N+2s}$. 
And from here
\begin{equation}\label{eq511}
\theta_{2j}(t)\leq t^{-1/\gamma}\theta_j^2(0), \;\;\;t\in [0,T].
\end{equation}
Iterating \eqref{eq511}, as in  \cite[Claim 5.5]{DD}, we conclude
$$\sup_\Omega w(x,t)\leq C t^{-1/2\gamma}\|{\overline{\chi}_1}\|_{L^2(\Omega)},$$
that is, we have obtained \eqref{objetivo2} with $\beta=1/2\gamma$ and $C_0=C_0(\|{\overline{\chi}_1}\|_{L^2(\Omega)})$.
Therefore \eqref{objetivo} holds with $c(t)=C_0 e^{\lambda_1 t}t^{-1/2\gamma}$.
\end{proof}

\noindent Following the ideas developed in \cite[Lemma 2]{martel} we present now the last result needed to prove Theorem \ref{main_parabolic}. 
\begin{prop} \label{Martel}
Let $u$ be the solution to \eqref{probP}. Then,
$$u(x,t)\geq  c(t)\left(\int_{\Omega}{u_0(x)\delta^s(x)}\right)\delta^s(x),\quad x\in\Omega,\quad t\in[0,T],$$
where $\delta(x):=\textrm{dist}(x,\partial\Omega)$ and  $c(t)> 0$.
\end{prop}

\begin{proof}

Let $\{S(t)\}_{t\geq 0}$ be the analytic heat semigroup with zero mixed conditions. Therefore for every $x_0\in B\subseteq\overline{\Omega}$ by Hopf's Lemma, for every $t\in[0,T]$ we get that
$$u(x_0,t/2)=S(t/2) u_0(x_0)\!=\!\int_{\Omega}u_0 (x)S(t/2)\delta_{x_0}(x)\, dx\!\geq\!c_0(t)\!\int_{\Omega}u_0(x)\delta^{s}(x)\, dx,$$
where $\delta_{x_0}$ is the Dirac distribution in $x_0$. That is,
\begin{equation}\label{dieciocho}
u(x,t/2)\geq c_0(t)\|u_0\delta^s\|_{L^{1}(\Omega)}\chi_{B},
\end{equation}
where $\chi_{B}$ is the characteristic function of the ball $B$. Consider now 
$$\mbox{$\widetilde{u}(x,t)$ the solution of \eqref{probP} with initial datum equal to $u(x,t/2)$}$$
and 
$$\mbox{$\bar{u}(x,t)$ the solution of \eqref{probP} with initial datum equal to $\chi_{B}$.}$$ 
Then by \eqref{dieciocho} and the comparison principle it follows that
$$\widetilde{u}(x,t/2)\geq c(t)\|u_0\delta^s\|_{L^{1}(\Omega)} \bar{u}(x,t/2),\, x\in\Omega, t\in[0,T],\, c(t){>} 0.$$
Thus, since by the property of semigroup we have that
$$u(x,t)=S(t)u_0(x)=S(t/2)u(x,t/2)=\widetilde{u}(x,t/2),$$
the previous inequality and the Hopf's Lemma imply
$$u(x,t)\geq c(t)\|u_0\delta^s\|_{L^{1}(\Omega)}\delta^{s}(x),$$
for every $x\in\Omega$ and $t\in[0,T]$ as desired.
\end{proof}

\noindent We can now conclude the 
\begin{demoparabolic}
Looking carefully at the proof of  Theorem \ref{key_parabolico}, we deduce that if $u\geq 0$ solves \eqref{probP} and satisfies $u(t)\geq c(t)\delta^{s}(x)$ for $0<t<T$ then
$$u(x,t)\geq C_0 e^{\lambda_1 t}t^{-1/2\gamma}\overline{\xi_0}.$$
Thus following verbatim the proof of \cite[Corollary 2.8]{DD}, by Proposition \ref{Martel}, the estimate \eqref{estimacion_parabolica} follows.
\end{demoparabolic}

\end{document}